\documentclass[a4paper, 9pt]{article}
\usepackage{graphics}
\usepackage[dvips]{color}

\usepackage{lscape}
\usepackage{latexsym}
\usepackage{amsmath,verbatim}
\usepackage{graphics}
\usepackage{amsthm}
\usepackage{amssymb}
\usepackage[mathscr]{euscript}
\usepackage{yfonts}
\usepackage{makeidx}
\usepackage{stmaryrd}
\usepackage{multicol}
\usepackage{bm}
\usepackage[all]{xy}
\numberwithin{equation}{section}
\newtheorem{thm}{Theorem}[section]
\newtheorem{prop}[thm]{Proposition}
\newtheorem{lem}[thm]{Lemma}
\newtheorem{cor}[thm]{Corollary}
\newtheorem{claim}{Claim}{\bf}{\it}

\newtheorem{fthm}{Theorem}{\bf}{\it}
{\bf}{\it}
\newtheorem{fcor}[fthm]{Corollary}{\bf}{\it}
\theoremstyle{definition}
\newtheorem{defn}[thm]{Definition}

{\bf}{\rm}

\theoremstyle{remark}

\newtheorem{rem}[thm]{Remark}
{\bf}{\it}

\newtheorem{definition and corollary}[thm]{Definition and Corollary}

{\it}{\rm}

\newcommand{\A}{{\mathbb A}}
\newcommand{\hB}{\widehat{B}}
\newcommand{\al}{\alpha}

\newcommand{\C}{{\mathbb C}}
\newcommand{\cO}{{\mathcal O}}

\newcommand{\Hom}{\mbox{\rm Hom}}

\newcommand{\bi}{{\mathbf i}}

\newcommand{\la}{\lambda}

\newcommand{\Fr}{\mathsf{Fr}}
\newcommand{\g}{\mathfrak{g}}
\newcommand{\gb}{\mathfrak{b}}

\newcommand{\h}{\mathfrak{h}}

\newcommand{\tI}{\mathtt{I}}
\newcommand{\tJ}{\mathtt{J}}
\newcommand{\hN}{\widehat{N}}
\newcommand{\gn}{\mathfrak{n}}

\newcommand{\bO}{\mathbb{O}}

\renewcommand{\P}{\mathbb{P}}
\newcommand{\Proj}{\mathrm{Proj}}

\newcommand{\R}{\mathbb{R}}
\newcommand{\bX}{\mathbb{X}}
\newcommand{\gX}{\mathfrak{X}}

\newcommand{\gY}{\mathfrak{Y}}

\newcommand{\Z}{\mathbb{Z}}

\newcommand{\Span}{\mbox{\rm Span}}

\newcommand{\Ga}{\mathbb G_a}
\newcommand{\Gm}{\mathbb G_m}

\title{Frobenius splitting of thick flag manifolds of Kac-Moody algebras}
\author{Syu \textsc{Kato}\footnote{Department of Mathematics, Kyoto University, Oiwake Kita-Shirakawa Sakyo Kyoto 606-8502 JAPAN \tt{E-mail:syuchan@math.kyoto-u.ac.jp}}}

\begin{document}
\maketitle

\begin{abstract}
We explain that the Pl\"ucker relations provide the defining equations of the thick flag manifold associated to a Kac-Moody algebra. This naturally transplant the result of Kumar-Mathieu-Schwede about the Frobenius splitting of thin flag varieties to the thick case. As a consequence, we provide a description of the space of global sections of a line bundle of a thick Schubert variety as conjectured in Kashiwara-Shimozono [Duke Math. J. 148 (2009)]. This also yields the existence of a compatible basis of thick Demazure modules, and the projective normality of the thick Schubert varieties.
\end{abstract}

\section*{Introduction}
The geometry of flag varieties of a Lie algebra $\g$ is ubiquitous in representation theory. In case $\g$ is a Kac-Moody algebra, we have two versions of flag varieties $X$ and $\bX$, that we call the thin flag varieties and thick flag manifolds, respectively (see e.g. \cite{KT95}). They coincide when $\g$ is of finite type, and in this case we have
\begin{equation}
X = \bX = \mathrm{Proj} \, \bigoplus _{\la} L ( \la )^{\vee},\label{BW}
\end{equation}
where $\la$ runs over all dominant integral weights and $L ( \la )$ denotes the corresponding integrable highest weight representation of $\g$. The isomorphism (\ref{BW}) is less obvious when $\g$ is not finite type since $L ( \la )$ is no longer finite-dimensional. In fact, the both of $X$ and $\bX$ are quotients of certain Kac-Moody groups $G$ associated to $\g$, and we can ask whether we have 
\begin{equation}
G \cong \mathrm{Spec} \, \Bbbk [G]\label{PW}
\end{equation}
as an enhancement of (\ref{BW}), where $\Bbbk [G]$ is the coordinate ring of $G$ (cf. Kac-Peterson \cite{KP83}). However, Kashiwara \cite[\S 6]{Kas89} explains that none of the choice of $G$ can satisfy (\ref{PW}) for any version of a reasonably natural commutative ring $\Bbbk [G]$.

The goal of this paper is to explain that despite the above situation, we can still understand the geometry of Kac-Moody flag manifolds as infinite type schemes so that we can deduce some consequences in representation theory.

To explain what we mean by this, we introduce some more notation: The scheme $\bX$ admits a natural action of the subgroup $B$ of $G$ that corresponds to the non-negative part of $\g$, and the set of $B$-orbits of $\bX$ is in natural bijection with the Weyl group $W$ of $\g$. Hence, we represent a $B$-orbit closure of $\bX$ by $\bX ^w$ for some $w \in W$. For each integral weight $\la$ of $\g$, we have an associated line bundle $\cO_{\bX} ( \la )$ and its restriction $\cO_{\bX^w} ( \la )$ to $\bX^w$.

The main result in this paper is:

\begin{fthm}[$\doteq$ Theorem 1.23 and Corollary 2.21]\label{fP}
For an arbitrary Kac-Moody algebra, the thick flag manifold $\bX$ admits the presentation $(\ref{BW})$ as schemes. Similar result holds for each $B$-orbit closure of $\bX$.
\end{fthm}

As Kashiwara's embedding of $\bX$ into the Grassmannian \cite[\S 4]{Kas89} factors through a highest weight integrable module, Theorem \ref{fP} asserts that it is a closed embedding. Hence, Theorem \ref{fP} affirmatively answers the question in \cite[4.5.6--4.5.7]{Kas89}.

The thin flag variety $X$ forms a Zariski dense subset of $\bX$ (see e.g. Kashiwara-Tanisaki \cite[\S 1.3]{KT95}). This implies that the projective coordinate ring of $X$ (as an ind-scheme) is the completion of that of $\bX$ (as a honest scheme). Therefore, we can transplant the Frobenius splitting of $X$ (or its ind-pieces) to that of $\bX$ provided in Kumar-Schwede \cite{KS14}:

\begin{fcor}[$\doteq$ Corollary 2.12]\label{fFS}
For an arbitrary Kac-Moody algebra over an algebraically closed field $\Bbbk$ of positive characteristic, the thick flag manifold $\bX$ admits a Frobenius splitting that is compatible with the $B$-orbits.
\end{fcor}

From this, we deduce some conclusions on the level of global sections as:

\begin{fthm}[$\doteq$ Theorem 2.17, 2.19, and Corollary 2.21, 2.22]\label{Icor}
For each $w \in W$, we have:
\begin{enumerate}
\item the natural restriction map
\begin{equation}
\Gamma ( \bX,\cO_{\bX} ( \la )) \longrightarrow \!\!\!\!\! \rightarrow \Gamma ( \bX^w,\cO_{\bX^w} ( \la ))\label{frest}
\end{equation}
is surjective;
\item the image of the inclusion
$$\Gamma ( \bX^w,\cO_{\bX^w} ( \la ))^{\vee} \subset \Gamma ( \bX,\cO_{\bX} ( \la ))^{\vee} = L ( \la )^{\vee}$$
obtained as the dual of $(\ref{frest})$ is cyclic as a $\mathrm{Lie} \, B$-module;
\item the scheme $\bX^w$ is projectively normal;
\item the sums of modules in $\{\Gamma ( \bX^w,\cO_{\bX^w} ( \la ))^{\vee}\}_{w \in W}$ forms a distributive lattice in terms of intersection.
\end{enumerate}
\end{fthm}

We remark that Theorem \ref{Icor} 4) should be also obtained as a combination of Kashiwara's crystal basis theory \cite{Kas94} and Littelmann's path model theory \cite{Lit95} when $\g$ is symmetrizable. However, the only reference the author is aware beyond the finite case is the affine case presented in Ariki-Kreimann-Tsuchioka \cite[\S 6]{AKT08} (as stated there, the proofs of this part are due to Kashiwara and Sagaki).

We also note that Theorem \ref{Icor} 1) and 2) confirms a part of the Kashiwara-Shimozono conjecture \cite[Conjecture 8.10]{KS09} (that originally concerns when $\g$ is affine).

\section{Defining equations of thick flag manifolds}
We work over an algebraically closed field $\Bbbk$. We employ \cite{Kum02} as a basic reference, and we may refer to \cite{Kum02} also for $\mathrm{char} \, \Bbbk > 0$ case without a comment (while the book deals only for $\Bbbk = \C$) when we  supply enough (other) results so that its proof carries over based on them.

Let $\tI$ be a finite set with its cardinality $r$ and let $C = (c_{ij})_{i,j\in \tI}$ be a generalized Cartan matrix (GCM) in the sense of \cite[\S 1.1]{Kac}. Let $\g$ be the Kac-Moody algebra associated to $C$, and let $\h$ be its Cartan subalgebra (we have $\dim_{\Bbbk} \, \h = 2 |\tI| - \mathrm{rank} \, C$). Let $Q$ and $Q^{\vee}$ be the root lattice and the coroot lattice of $\g$, and $\{\al_i\}_{i \in \tI} \subset Q$ and $\{\al_i^{\vee}\}_{i \in \tI} \subset Q^{\vee}$ are the set of simple roots and the set of simple coroots, respectively. Let $X^*$ be a $\Z$-lattice that contains $Q$ and equipped with elements $\varpi_1,\ldots,\varpi_r \in X^*$ so that $X^* \otimes_{\Z} \Bbbk \cong \h^*$, and there exists a pairing
$$\left< \bullet, \bullet \right> : Q^{\vee} \times X^* \longrightarrow \Z$$
that satisfies
$$\left< \al_i^{\vee}, \al_j \right> = c_{ij}, \hskip 5mm \left< \alpha_i^{\vee}, \varpi_j \right> = \delta_{ij}, \hskip 5mm \text{and} \hskip 5mm \left< \alpha_i^{\vee}, X^* \right> = \Z.$$
Let $\{E_i, F_i\} _{i \in \tI}$ be the Kac-Moody generators of $\g$ so that $[E_{i},F_{j}] = \delta_{ij} \alpha_{i}^{\vee} \in \h$ for $i,j \in \tI$. Let $\gn, \gn^- \subset \g$ be the Lie subalgebras generated by $\{E_i\}_{i \in \tI}$ and $\{F_i\}_{i \in \tI}$, respectively. We set $H := \mathrm{Spec} \, \Bbbk [ e^{\la} \mid \la \in X^* ]$. We have $\mathrm{Lie} \, H = \h$. For each $\al \in X^*$, we define
$$\g _\al := \{ \xi \in \g \mid \mathrm{Ad} ( h ) \xi = \al ( h ) \xi, \hskip 3mm \forall h \in H \}, \hskip 3mm \mathrm{mult} \, \al := \dim \, \g_\al.$$
We set
$$\Delta^+ := \{ \al \in X^* \setminus \{ 0 \} \mid \g_\al \subset \gn \}, \hskip 3mm \Delta^- := - \Delta^+.$$

We have reflections $\{ s_i \}_{i \in \tI}$ on $\mathrm{Aut} ( X^* )$ that generates a Coxeter group $W$. We denote its length function by $\ell$, and the Bruhat order by $<$ (see Kumar \cite[Definition 1.3.15]{Kum02}). We have a subset
$$\Delta^+_{\mathrm{re}} := \Delta^+ \cap W \{ \al_i \}_{i \in \tI} \subset \Delta^+.$$
Each $\al \in \Delta^{+}_{\mathrm{re}}$ gives a reflection $s _{\al} \in W$ defined through the conjugation of a simple reflection. We have $\mathrm{mult} \, \al = 1$ for $\al \in \Delta^+_{\mathrm{re}}$, and we have $\mathfrak{sl} ( 2 ) \cong \g_\al \oplus \Bbbk \al^{\vee} \oplus \g_{-\al}$ as Lie algebras in this case.

For each $i \in \tI$, we define $\mathop{SL} ( 2, i )$ as the connected and simply connected algebraic group with an identification $\mathrm{Lie} \, \mathop{SL} ( 2, i ) = \Bbbk E_i \oplus \Bbbk \alpha_i^{\vee} \oplus \Bbbk F_i$.
For each $n > 0$, we set
$$\Delta^- ( n ) := \{ \al \in \Delta^- \mid \al = - \sum_{i \in \tI} m_i \al_i, \hskip 2mm m_i \in \Z_{\ge 0}, \hskip 2mm \sum_i m_i \le n \} \subset \Delta^-.$$
Then, $\bigoplus_{\al \in \Delta^- \backslash \Delta ^- ( n )} \g_{-\al} \subset \gn$ and $\bigoplus_{\al \in \Delta^- \backslash \Delta ^- ( n )} \g_{\al} \subset \gn^-$ are ideals. We denote the quotients by $\gn ( n )$ and $\gn ^- ( n )$, respectively. By construction, we have a Lie algebra quotient maps $\gn ( n ) \rightarrow \gn ( n' )$ and $\gn^- ( n ) \rightarrow \gn^- ( n' )$ for $n > n'$.

We define a pro-unipotent group
$$\hN ^- := \varprojlim_n N^- ( n ),$$
where
\begin{itemize}
\item $N^- ( n )$ is a smooth connected unipotent algebraic group with its Lie algebra $\gn ^- ( n )$ for each $n \in \Z_{> 0}$;
\item the transition maps in the inverse limit are surjective smooth morphisms that induce the Lie algebra quotients above.
\end{itemize}
This group, together with the groups $N$,  $N ( H )$, $G^-$ defined below, and a lift $S$ of $\{s_i\}_{i \in \tI} \subset W$ to $N ( H )$, must satisfy the axioms presented in \cite[Definition 5.2.1]{Kum02} as a 6-tuple $(G^-,N ( H ), \hN^-, N, H, S)$. Applying the Chevalley involution to $\{ N^- ( n ) \}_{n \ge 1}$, we obtain a pro-unipotent group $\hN := \varprojlim_n N ( n )$ corresponding to $\gn$.

We define $\hB^+ := H \hN$ and $\hB^- := H \hN^-$, that are (pro-algebraic) groups (and also a Lie subalgebra $\mathfrak b := \mathfrak h \oplus \mathfrak n \subset \mathfrak g$). For each $\al \in \Delta^+_{\mathrm{re}}$, we have a one-parameter unipotent subgroup $\rho_\al : \Ga \rightarrow \hB^{+}$ so that $h \rho_\al ( z ) h^{-1} = \rho_\al ( \al ( h ) z )$ for every $z \in \Ga$ and $h \in H$. Similarly, we have a one-parameter unipotent subgroup $\rho_{-\al} : \Ga \rightarrow \hB^{-}$.

We have subgroups $N^+ \subset \widehat{N}^+$ and $N^- \subset \widehat{N}^-$ formed by products of finitely many elements from $\{ \rho_{\al_i} ( \mathbb G_a ) \}_{i \in \tI}$ and $\{ \rho_{- \al_i} ( \mathbb G_a ) \}_{i \in \tI}$, respectively. Let $N ( H )$ denote the group generated by $H$ and the normalizers of $H$ inside $\mathop{SL} ( 2, i )$ for each $i \in \tI$, whose quotient by $H$ is $W$. We have a translation of elements of $\hB^{\pm}$ under the action of $N ( H )$, defined partially (see \cite[\S 6.1]{Kum02}). The positive Kac-Moody group $G^+$ is defined as the amalgamated product of $\hB^+$ and $N ( H )$, while the negative Kac-Moody group $G^-$ is defined as the amalgamated product of $\hB^-$ and $N ( H )$ (see \cite[\S 5.1]{Kum02}). For each $\tJ \subset \tI$, we have a partial amalgam $\hB^{\pm} \subset \hB^{\pm} _{\tJ} \subset G^{\pm}$, that we call the parabolic subgroups corresponding to $\tJ$. 

Let $U_\Z ( \g )$ (resp. $U_\Z ( \h )$, $U_\Z ( \gb )$ or $U_\Z ( \gn^- )$) be the Chevalley-Kostant $\Z$-form of the enveloping algebra of $\g$ generated by $E_i^{(n)}, F_i^{(n)}$ $(i \in \tI, n \in \Z_{> 0})$ and
$$h ( m ) := \frac{h ( h - 1 ) \cdots ( h - m + 1 )}{m!} \hskip 5mm h \in \mathrm{Hom}_{\Z} ( X^*, \Z ), m \in \Z_{\ge 0}$$
(resp. $h ( m )$, $E_i^{(n)}$ and $h ( m )$, or $F_i^{(n)}$), and let $U ( \g )$ (resp. $U ( \h )$, $U ( \gb )$, or $U ( \gn^- )$) be its specialization to $\Bbbk$ (see e.g. Tits \cite{Tit82} or Mathieu \cite[Chapter I]{Mat88}).

We understand that a representation of an algebraic group is always algebraic. Note that the complete reducibility of representations always hold for split torus (and we never deal with non-split torus in this paper).

\begin{defn}[integrable highest weight modules]
A $( U ( \h ), H )$-module $M$ is said to be a weight module if $M$ admits a semi-simple action of the above $h ( m )$'s that integrates to the algebraic $H$-action. In this case, we denote by $M_\mu \subset M$ the $H$-weight space of weight $\mu \in X^*$. We call $M$ restricted if we have $\dim \, M_\la < \infty$ for every $\la \in X^*$.\\
A $( U ( \g ), H )$-module $M$ is said to be a highest weight module if $M$ is a weight module as a $( U ( \h ), H )$-module and $M$ carries a cyclic $U ( \g )$-module generator that is a $( U ( \gb ), H )$-eigenvector.\\
A $( U ( \g ), H )$-module $M$ is said to be an integrable module if it is a restricted weight module, we have $\dim \, \Span_\Bbbk \{ E_i^{(n)}F_i^{(m)} v \} < \infty$ for each $v \in M$ and $i \in \tI$, and it integrates to an algebraic $\mathop{SL} ( 2, i )$-action that is compatible with the $H$-action.
\end{defn}

\begin{lem}
Let $M$ be an integrable $U ( \g )$-module. Then, every $U ( \g )$-submodule of $M$ is again integrable.\hfill $\Box$
\end{lem}

\begin{thm}[Mathieu \cite{Mat88,Mat89}]\label{pairing}
We have a non-dengenerate $\Bbbk$-linear pairing
$$(\bullet, \bullet) : U ( \gn^- ) \otimes \Bbbk [\hN^-] \ni ( P, f ) \mapsto ( P f ) ( 1 ) \in \Bbbk.$$
\end{thm}

\begin{proof}
Note that $U (\gn^-)$ is equipped with a restricted $( U ( \h ), H )$-module structure arising from the adjoint action of $H$. Hence, the (restricted) $\Bbbk$-dual of $U ( \gn^- )$ is well-defined. Moreover, the natural Hopf algebra structure of $U ( \gn^- )$ (so that $\Delta ( F_i ^{(n)} ) = \sum_{m=0}^n F_i ^{(m)} \otimes F_i ^{(n-m)}$ for each $i \in \tI$ and $n \ge 0$) induces a commutative bialgebra structure on $U ( \gn^- )^{\vee}$. Then, $\mathrm{Spec} \, U ( \gn^- )^{\vee}$ is the pro-algebraic group associated to $\gn^-$ in \cite{Mat88, Mat89} by \cite[Lemme 2]{Mat89}.

By \cite[Lemme 3]{Mat89}, $\mathrm{Spec} \, U ( \gn^- )^{\vee}$ satisfies the conditions on $\hN^-$ listed above. By replacing the arguments in \cite[\S 6.1]{Kum02} involving the exponential maps to our pro-unipotent group structures of $\hN^-$ and unipotent one-parameter subgroups $\{ \rho_{\al} \}_{\al}$, we deduce that $(G^-,N ( H ), \hN^-, N, H, S)$ satisfies the conditions in \cite[Definition 5.2.1]{Kum02} by \cite[Theorem 6.1.17]{Kum02} and its proof.
\end{proof}

We define $P := \bigoplus_{i \in \tI} \Z \varpi_i$, $P_+ := \bigoplus_{i \in \tI} \Z_{\ge 0} \varpi_i$, and $P_{++} := \bigoplus_{i \in \tI} \Z_{\ge 1} \varpi_i$. For $\tJ \subset \tI$, we set $P_+ ^{\tJ}:= \bigoplus_{i \in \tI \backslash \tJ} \Z_{\ge 0} \varpi_i$. For each $\la \in P$, we have a Verma module $M ( \la )$ defined as:
$$M ( \la ) = U ( \g ) \otimes _{U ( \gb )} \Bbbk _{\la}.$$
The Verma modules are restricted weight modules and are generated by a unique vector $v_\la$ with $H$-weight $\la$. We define
$$L ( \la ) := M ( \la ) / \sum_{i \in \tI} U ( \g ) F_i^{(\left< \alpha_i^{\vee}, \la \right> + 1)} v_\la.$$
\begin{lem}\label{maxint}
For each $\la \in P_+$, the module $L ( \la )$ is the maximal integrable quotient of $M ( \la )$.
\end{lem}

\begin{proof}
The assertion is \cite[Lemma 2.1.7]{Kum02} when $\mathrm{char} \, \Bbbk = 0$. Its proof also asserts that every integrable module is a quotient of $L ( \la )$ for $\mathrm{char} \, \Bbbk > 0$ (as an effect of our definition of integrality). For each $i \in \tI$, the module
$$M ( \la ) / U ( \g ) F_i^{(\left< \alpha_i^{\vee}, \la \right> + 1)} v_\la$$
is $\mathop{SL} ( 2, i )$-integrable (recall that the pro-unipotent radical of the parabolic subgroup corresponding to $i \in \mathtt I$ is $\mathop{SL} ( 2, i )$-stable by the GCM condition $\left< \al_i^{\vee}, \al_j \right> \le 0$ when $i \neq j$ \cite[\S 1.1]{Kac}), and it is the maximal $\mathop{SL} ( 2, i )$-integrable quotient of $M ( \la )$. Hence, we deduce that $L ( \la )$ is the maximal integrable quotient of $M ( \la )$ as required.
\end{proof}

\begin{cor}
For each $\la \in P_+$, the $H$-character of $L ( \la )$ obeys the Weyl-Kac character formula.
\end{cor}

\begin{proof}
This is \cite[Theorem 8.3.1]{Kum02} when $\mathrm{char} \, \Bbbk = 0$. When $\mathrm{char} \, \Bbbk > 0$, the arguments in \cite{Mat88} asserts that some integrable submodule of $L ( \la )$ obtained as a successive application of Demazure-Joseph functors obeys the Weyl-Kac character formula. As such a submodule contains $v_{\la}$, it must be the whole $L (\la)$.
\end{proof}

For each $w \in W$ and $\la \in P_+$, we have a unique non-zero vector $v_{w\la} \in L ( \la )$ of weight $w \la$ up to scalar. We define the thin Demazure module and thick Demazure module as:
$$L _w ( \la ) := U ( \gn ) v_{w \la}, \hskip 5mm L ^w ( \la ) := U ( \gn^- ) v_{w \la} \hskip 5mm \subset L ( \la ).$$
These admit $H$-eigenspace decompositions.

We define the tensor product of two restricted weight modules $M, N$ as:
$$M \otimes N := \bigoplus_{\la, \mu \in X^*} M_{\la} \otimes N_{\mu}.$$
We define the dual of a restricted weight module $M$ as:
$$M^{\vee} := \bigoplus_{\la \in X^*} M_\la^*,$$
for which the natural inclusion $M^{\vee} \subset M^*$ defines a $H$-submodule. The completion of a restricted weight module $M$ is defined as:
$$M^{\wedge} := \prod_{\la \in X^*} M_\la.$$
It is straight-forward to see that if $M$ admits a Lie algebra action that contains $\h$ whose action prolongs to the $H$-action, then so are $M^{\vee}$ and $M^{\wedge}$. Note that the $H$-action on $M^{\vee}$ is $H$-finite.

\begin{defn}[thin flag varieties; \cite{Kum02} \S 7.1]
The thin flag variety $X$ is defined set-theoretically as $G^+ ( \Bbbk ) / \widehat{B}^+ ( \Bbbk )$, and the generalized thin flag variety $X_{\tJ}$ for $\tJ \subset \tI$ is defined set-theoretically as $G^+ ( \Bbbk ) / \widehat{B}^+_{\tJ} ( \Bbbk )$. In particular, we have $X = X_{\emptyset}$. Their indscheme structures are given through an embedding into $\cup_{w \in W} \P ( L_w ( \la ) )$ ($=\P ( L ( \la ) )$) for $\la \in P_+^{\tJ}$ (see \cite[\S 7]{Kum02}). For each $w \in W$, we set $X_w := X \cap \P ( L_w ( \la ) )$ and $X_{w, \tJ} := X_{\tJ} \cap \P ( L_w ( \la ) )$, and call them the thin Schubert variety and the generalized thin Schubert variety, respectively.
\end{defn}

\begin{rem}
The (ind-)scheme structures of $X, X_{\tJ}, X_w, X_{w,\tJ}$ are independent of the choice of $\la$ (see \cite[Theorem 7.1.15 and Remark 7.1.16]{Kum02} and Mathieu \cite[Corollaire 2]{Mat89}).
\end{rem}

By \cite[Proposition 7.1.15]{Kum02}, we know that $X = \bigcup_w X_w$ and $X_{\tJ} = \bigcup_w X_{w,\tJ}$.

We have an embedding $L ( \la ) \subset L ( \la )^{\wedge}$ for $\la \in P_+$. The group $G^{-}$ acts on $L ( \la )^{\wedge}$, while the group $G^+$ acts on $L ( \la ) \subset L ( \la )^{\wedge}$. Let $\mathbb O^e$ be the $\hN^-$-orbit of $[v_{\la}]$ in $\P ( L ( \la )^{\wedge} )$, whose scheme structure is independent of the choice of $\la \in P_{++}$.

\begin{defn}[thick flag manifolds; \cite{Kas89} \S 5.8]
The thick flag manifold $\bX'$ is defined set-theoretically as the union of $N ( H ) ( \Bbbk )$-translates of $\mathbb O^e ( \Bbbk )$.
\end{defn}

\begin{rem}
In the following, we only need to use the fact that the scheme structure of the thick flag manifold $\bX'$ given in \cite{Kas89} has $\bX' ( \Bbbk ) = N ( H ) ( \Bbbk ) \cdot \mathbb O^e ( \Bbbk )$ as its set of $\Bbbk$-valued points, it admits the $G^-$-action, and it has $\mathbb O^e$ as its $\hB^-$-stable (affine open) subscheme. Note that we have $\mathbb O^e\cong \hN^-$, and its scheme structure is the same as these transported from $\P ( L ( \la )^{\wedge} )$ (for every choice of $\la \in P_{++}$) or the Grassmannian employed in \cite{Kas89}. 
\end{rem}

\begin{rem}
Assume $\g$ {\it not} to be of finite type. By construction, we easily find an inclusion $X \subset \bX'$. This inclusion cannot be an equality as the dimension of $X$ is countable, while the dimension of $\bX'$ is uncountable. In general, $X$ is not smooth (\cite{FGT}, but can be formally smooth \cite{Zhu17}), while $\bX'$ is always smooth (in the sense it is a union of affine spaces) by construction.
\end{rem}

Theorem \ref{pairing} identifies $M ( \la )^{\vee}$ with a rank one $\Bbbk [\hN^-]$-module. This also induces an inclusion
$$L ( \la ) ^{\vee} \hookrightarrow M ( \la )^{\vee} \cong \Bbbk [\hN^-] \otimes _{\Bbbk} \Bbbk _{- \la} \hskip 5mm \la \in P_+.$$
Note that $M ( \la )^{\vee}$ naturally admits an action of $U ( \g )$, with a unique cocyclic $H$-eigenvector of weight $- \la$. Hence, we have an inclusion
\begin{equation}
\bigoplus_{\la \in P_+} L ( \la ) ^{\vee} \subset \bigoplus_{\la \in P_+} M ( \la ) ^{\vee} \cong \Bbbk [\hN^-] \otimes \bigoplus_{\la \in P_+} \Bbbk_{-\la} \subset \Bbbk [\hB^-],\label{inclmod}
\end{equation}
where the RHS is a commutative ring.

\begin{lem}\label{incl}
For each $\lambda, \mu \in P_+$, we have a unique $U ( \g )$-module morphism $($up to a scalar$)$
$$m_{\la, \mu} : L ( \la )^{\vee} \otimes L ( \mu ) ^{\vee} \longrightarrow L ( \lambda + \mu )^{\vee}$$
that makes $\bigoplus_{\la \in P_+} L ( \la )^{\vee}$ into an integral commutative subring of $\Bbbk [\hN^-]$. Moreover, the map $m_{\la, \mu}$ is surjective for every $\la, \mu \in P_+$, and the ring $\bigoplus_{\la \in P_+} L ( \la )^{\vee}$ is generated by $\bigoplus _{i \in \tI} L ( \varpi_i )^{\vee}$.
\end{lem}

\begin{proof}
By the comparison of the defining equation, we have a unique $U ( \g )$-module map (up to scalar)
$$m_{\la, \mu}^{\vee} : L ( \lambda + \mu ) \rightarrow L ( \la ) \otimes L ( \mu )$$
that respects the $H$-weight decomposition. By taking the dual, we obtain the desired map. Each $L ( \la )$ is a quotient of $M ( \la )$, and we have an isomorphism
$$M ( \la )^{\vee} \cong \Bbbk [\hN^-] \otimes _{\Bbbk} \Bbbk _{- \la}$$
as $U ( \gn^- )$-modules. The $HN^-$-equivariant multiplication of $\Bbbk [\hB^-]$ is uniquely determined by that of the $N^-$-fixed elements, that is $\Bbbk [X^*]$. This forces $L ( \la )^{\vee} \cdot L ( \mu )^{\vee} \subset M ( \la + \mu )^{\vee}$ inside $\Bbbk [\hB^-]$. Since the tensor product of integrable modules is integrable, we deduce that $L ( \la )^{\vee} \cdot L ( \mu )^{\vee} \subset L ( \la + \mu )^{\vee}$ inside $\Bbbk [\hB^-]$. Therefore, the inclusion (\ref{inclmod}) respects the product structure (uniquely) induced by $m_{\la,\mu}^{\vee}$. The resulting ring is commutative and integral by \cite[Lemme 2]{Mat89}, and its multiplication maps are surjective by \cite[Corollaire 2]{Mat89}.

The commutativity of the product and the integrality of $\bigoplus_{\la \in P_+} L ( \la )^{\vee}$ can be also deduced from these of $\Bbbk [\hB^-]$ (though our Theorem \ref{pairing} depends on these facts through \cite[Lemme 2]{Mat89} unless we employ the theory of global base \cite{Kas91, Kas94} to prove it by additionally assuming $\g$ is symmetrizable).
\end{proof}

\begin{defn}\label{mproj}
Let $\tJ \subset \tI$. For a $P_+^{\tJ}$-graded ring $R = \bigoplus_{\la \in P_+^{\tJ}} R_{\la}$ with $R_0 = \Bbbk$ that is generated by $\bigoplus_{i \in \tI\setminus\tJ} R_{\varpi_i}$, we define $\mathrm{Proj}_{\tJ} \, R$ to be
$$\mathrm{Proj}_{\tJ} \, R := \left( \mathrm{Spec} \, R \setminus \{ x \in \mathrm{Spec} \, R \mid x \not\equiv 0 \text{ on } R_{\varpi_i} \hskip 2mm \forall i \in \tI \setminus \tJ \} \right) / H,$$
where $H$ acts on $R_{\varpi_i}$ through the character $\varpi_i$ for each $i \in \tI$. We might drop subscript $\tJ$ when the meaning is clear from the context.
\end{defn}

\begin{rem}
We note that our condition guarantees $\mathrm{Proj}_{\tJ} \, R \subset \prod_{i\in\tI \setminus \tJ} \P ( R_{\varpi_i} ^* )$, that in turn implies that $P_+^{\tJ}$ is in the closure of the ample cone of $\mathrm{Proj}_{\tJ} \, R$.
\end{rem}

We denote the ring $\bigoplus_{\la \in P_+} L ( \la )^{\vee}$ in Lemma \ref{incl} by $R$. We define
$$\bX := \Proj \, R.$$
Note that each $\mathop{SL} ( 2, i )$ $(i \in I)$ and $H$ acts on $L ( \la )^{\vee}$, and hence on $\bX$. Hence, we derive an action of $N ( H )$ on $\bX$. By construction, we have a line bundle $\cO_{\bX} ( \la )$ on $\bX$ for each $\la \in P$. 

\begin{cor}\label{tmult}
For each $w \in W$ and $\la, \mu \in P_+$, the multiplication map $m_{\la, \mu}$ of $R$ induces a $U ( \gn ) $-module map and a $U ( \gn^- )$-modules map
$$m'_{\la, \mu} : L _w ( \la )^{\vee} \otimes L _w ( \mu ) ^{\vee} \to L _w ( \lambda + \mu )^{\vee}, \hskip 3mm m''_{\la, \mu} : L ^w ( \la )^{\vee} \otimes L ^w ( \mu ) ^{\vee} \to L ^w ( \lambda + \mu )^{\vee}$$
that are surjective and associative.
\end{cor}

\begin{proof}
By the dual of Lemma \ref{incl}, we have $L ( \la + \mu ) \subset L ( \la )\otimes L ( \mu )$.

By $m _{\la, \mu} ^{\vee} ( v_{w ( \la + \mu )} ) = v_{w \la} \otimes v_{w \mu}$, we deduce that the inclusion $L ( \la + \mu ) \subset L ( \la )\otimes L ( \mu )$ yields inclusions $L_w ( \la + \mu ) \subset L_w ( \la )\otimes L_w ( \mu )$ and $L^w ( \la + \mu ) \subset L^w ( \la )\otimes L^w ( \mu )$. Hence, the multiplication map $m_{\la, \mu}$ induce well-defined surjective maps
$$m'_{\la, \mu} : L _w ( \la )^{\vee} \otimes L _w ( \mu ) ^{\vee} \to L _w ( \lambda + \mu )^{\vee}, \hskip 3mm m''_{\la, \mu} : L ^w ( \la )^{\vee} \otimes L ^w ( \mu ) ^{\vee} \to L ^w ( \lambda + \mu )^{\vee}$$
that define quotient rings of $R$ (and hence they are associative).
\end{proof}

For each $w \in W$, we have two commutative algebras:
$$R^w := \bigoplus_{\la \in P_+} L^w ( \la )^{\vee}, \hskip 3mm \text{and} \hskip 3mm R_w := \bigoplus_{\la \in P_+} L_w ( \la )^{\vee},$$
whose multiplications are given in Corollary \ref{tmult}.

We have a natural $G^+$-equivariant line bundle $\cO _{X_w} ( \la )$ for each $w \in W$ and $\la \in P_+$, and we have a natural $G^+$-equivariant line bundle $\cO _{X_{w,\tJ}} ( \la )$ for each $w \in W$ and $\la \in P_+^{\tJ}$ (cf. \cite[\S 7.2]{Kum02}).

\begin{thm}[Mathieu \cite{Mat88} Th\'eor\`eme 3, cf. \cite{Kum02} Theorem 8.2.2]\label{thin-coh}
For each $\la \in P_+$, we have
$$H^i ( X_w, \cO _{X_w} ( \la ) ) \cong \begin{cases} L_w ( \la )^{\vee} & (i = 0) \\\{0\} & (i > 0)\end{cases}.$$
The analogous assertion holds for generalized thin Schubert varieties corresponding to $\tJ \subset \tI$ for every $\la \in P_+^{\tJ}$. \hfill $\Box$
\end{thm}

\begin{cor}\label{thin-mproj}
For each $w \in W$, we have $X_w = \mathrm{Proj} \, R _w$. The analogous assertion holds for generalized thin Schubert varieties corresponding to $\tJ \subset \tI$ by setting
$$R _{w, \tJ} := \bigoplus_{\la \in P_+^{\tJ}} L_w ( \la )^*.$$
\end{cor}

\begin{proof}
Combine Theorem \ref{thin-coh} and the fact that $X_{w, \tJ} \subset \P ( L _w ( \la ) )$ is a closed immersion for each $\la \in P^{\tJ}_+$ so that $\left< \al_i^{\vee}, \la \right> > 0$ for each $i \in \tI \setminus \tJ$.
\end{proof}

Thanks to Corollary \ref{thin-mproj}, we have an embedding $X_w \subset \bX$ for each $w \in W$. This particularly implies $ \bigcup _w X_w = X \subset \bX$.

\begin{lem}\label{fixed}
The set of $H$-fixed points of $\bX$ is in bijection with $W$. 
\end{lem}

\begin{proof}
A $H$-fixed point $x$ of $\bX$ gives a collection of non-zero $H$-eigenvectors $\{ v'_\la \}_{\la \in P_+} \in \prod_{\la \in P_+} L ( \la )$ so that $m_{\la,\mu} ^{\vee} ( v'_{\la + \mu} ) = v'_\la \otimes v'_\mu$ for $\la, \mu \in P_+$ by Lemma \ref{incl}. By Theorem \ref{thin-coh}, there exists $w \in W$ so that $x \in X_w$. It follows that
$$\bigcup _{w \in W} X_w ^H = \bX^H.$$
The set of $H$-fixed points of $X_w ^H$ is in common among all characteristic and is a subset of the translation of $\{ [v_{\la}] \}_{\la \in P_+}$ by $N ( H )$ that descends to $W$ (see \cite[\S 7.1]{Kum02}). Therefore, we conclude that $\bX ^H$ is in bijection with $W$.
\end{proof}

Let $x_w$ denote the $H$-fixed point of $X_w^H \subset \bX$ corresponding to the cyclic $H$-eigenvectors of $\{L_w ( \la )\}_{\la \in P_+}$ for each $w \in W$. By examining the stabilizer, we deduce an isomorphism
$$\hB^- x_w = \hN^- x_w \cong \A^{\infty} \hskip 3mm \text{for each} \hskip 2mm w \in W$$
inside $\prod_{\la \in P_+} \P ( L ( \la )^{\wedge} ) = \prod_{\la \in P_+} \P ( L ( \la )^{\vee,*} )$. We set $\bO^w := \hB^- x_w$ ($= \hN^- x_w$). It is easy to see that $\bO^e$ here is isomorphic to $\bO^e$ employed in the definition of $\bX'$ as a $\hB^-$-homogeneous space.

We denote $\hN^+ x_w = N^+ x_w \subset \bX$ by $\bO_w$.

\begin{prop}\label{std}
We have an inclusion $\bO^{e} \subset \bX$ obtained by inverting finitely many rational functions on $\bX$. In other words, $\bO^e$ is a standard open set of $\bX$ in the terminology of \rm{\cite{EGAI}}.
\end{prop}

\begin{proof}
By (\ref{inclmod}), inverting the unique $H$-weight $- \la$ vector $v_{\la}^* \in L ( \la )^{\vee}$ (up to scalar) yields
$$\sum _{\la \in P_+} ( v_{\la}^* ) ^{-1} L ( \la )^{\vee} \cong U ( \gn^- )^{\vee} \cong \Bbbk [\widehat{N}^-]$$
as algebras, where the second isomorphism is through the Hopf algebra structure of $U ( \gn^- )$. We can rearrange $\{v_{\la}^*\}_{\la \in P_+}$ so that it is closed under the multiplication. It follows that
$$\bO^e = \bX \backslash \{ v_{\varpi_i}^* = 0 \}_{i \in \tI}$$
as required.
\end{proof}

Proposition \ref{std} asserts that we have an inclusion $\mathbb O^e \subset \bX$ with a $\hB^-$-action extending the $N^-$-action on $\bX$. By using the $\mathop{SL} ( 2, i )$-actions for every $i \in \tI$, we deduce an action of $\hB^-$ (and hence the $G^-$-action) on $\bX$ extending the $N^-$-action. We set $\bX^w := \overline{\bO^w} \subset \bX$ and call it the thick Schubert variety corresponding to $w \in W$.

\begin{lem}\label{Z-dense}
The ind-scheme $X$ is Zariski dense in $\bX$.
\end{lem}

\begin{proof}
Since we have $L ( \la ) = \bigcup_{w \in W} L _w ( \la )$ for each $\la \in P_+$ (\cite[Lemma 8.3.3]{Kum02}), the regular functions on $\bX$ can be distinguished on $X$.
\end{proof}

\begin{thm}\label{tBW}
For each $\la \in P_+$, we have
$$H^0 ( \bX, \cO_{\bX} (\la) ) \cong L ( \la )^{\vee}.$$
\end{thm}

\begin{proof}
We first prove the first assertion. By Lemma \ref{Z-dense}, we have
$$H^0 ( \bX, \cO_{\bX} (\la) ) \cong \Gamma ( X, \cO_{\bX} (\la) ).$$
This induces an injective map
$$H^0 ( \bX, \cO_{\bX} (\la) ) \subset \varprojlim_{w} H^0 ( X_w, \cO_{X_w} (\la) ).$$
By Theorem \ref{thin-coh} (or directly from \cite[Corollary 8.3.12]{Kum02}; see also the proof of Lemma \ref{B-canonical}), we have $\varprojlim_{w} H^0 ( X_w, \cO_{X_w} (\la) ) \cong L ( \la )^*$. Therefore, we conclude
$$H^0 ( \bX, \cO_{\bX} (\la) ) \subset L ( \la )^*$$
as $\g$-modules. Here we have
$$H^0 ( \bX, \cO_{\bX} (\la) ) \hookrightarrow H^0 ( \mathbb O^e, \cO_{\bX} (\la) ) \cong M ( \la )^{\vee}.$$
In particular, $H^0 ( \bX, \cO_{\bX} (\la) )$ is $H$-semisimple, and hence we deduce
$$H^0 ( \bX, \cO_{\bX} (\la) ) \subset L ( \la )^{\vee} = L ( \la )^* \cap M ( \la )^{\vee} \subset M ( \la )^*.$$
By examining the ring $R$, we deduce that $L ( \la )^{\vee} \subset H^0 ( \bX, \cO_{\bX} (\la) )$. This forces
$$H^0 ( \bX, \cO_{\bX} (\la) ) \cong L ( \la )^{\vee}$$
as required.
\end{proof}

\begin{thm}[cf. \cite{Kas90}]\label{tBWX}
For each $\la \in P_+$, we have
$$H^0 ( \bX', \cO_{\bX'} (\la) ) \cong L ( \la )^{\vee}.$$
If we assume $\mathrm{char} \, \Bbbk = 0$ in addition, then we have 
$$H^{>0} ( \bX', \cO_{\bX'} (\la) ) \cong \{ 0 \}.$$
\end{thm}

\begin{proof}
Since $\bX'$ is the $G^-$-translate of $\mathbb O^e$, we have
$$H^0 ( \bX', \cO_{\bX'} (\la) ) \subset H^0 ( \mathbb O^e, \cO_{\bX'} (\la) ) \cong M ( \la )^{\vee}.$$
Let $\mathbb U \subset \bX'$ be a $\hB^-$-stable open subset. By $\mathop{SL} ( 2 )$-consideration, imposing the regularity conditions on a section of $H^0 ( \mathbb U, \cO_{\bX'} (\la) )$ along $\mathop{SL} ( 2, i ) \mathbb U$ is equivalent to impose the $\mathop{SL} ( 2, i )$-finiteness. We know that $G^-$ is topologically generated by $\mathop{SL} ( 2, i )$ for all $i \in \tI$. Therefore, the maximal integrable submodule of $M ( \la )^{\vee}$ is exactly the space of global sections of $\cO_{\bX'} (\la)$. This proves the first assertion by Lemma \ref{maxint}.

Now we assume $\mathrm{char} \, \Bbbk = 0$ to consider the latter assertion. The case of symmetrizable $\g$ is \cite[Theorem 5.2.1]{Kas90}. The Kempf resolution presented in \cite[(8.6)]{KS09} is valid for arbitrary Kac-Moody algebras, as the differential between terms can be interpreted as a $\mathop{SL} ( 2 )$-calculation if one removes unnecessary strata. We have the BGG resolution for arbitrary Kac-Moody algebras \cite[\S 3]{HK07} not by changing the construction (see e.g. \cite[\S 9.2]{Kum02}) but by proving that the resulting homology group is integrable. Therefore, their comparison yields the second assertion in general.
\end{proof}

\begin{cor}\label{embfp}
We have an embedding $\bX' \hookrightarrow \prod_{\la \in P_+} \P ( L ( \la )^{\vee} )$ of schemes.
\end{cor}

\begin{proof}
The morphism exists by Theorem \ref{tBWX}. Since the morphism is an embedding on $\mathbb O^e$ and equivariant with respect to the $N ( H )$-action, we conclude that it is an embedding.
\end{proof}

\begin{thm}\label{id}
The scheme $\bX$ is isomorphic to the thick flag manifold $\bX'$. 
\end{thm}

\begin{proof}
We borrow some notation from the proof of Proposition \ref{std}. By Corollary \ref{embfp}, we have $\bX' := \bigcup _{w \in W} \mathbb O^w \subset \bX$. We set $E := \bX \backslash \bX'$. 

It suffices to show $E = \emptyset$. Thanks to Proposition \ref{std}, the set $E$ is contained in the locus that $v_{\la}^* = 0$ for some $\la \in P_+$. Note that $E$ admits natural $\mathop{SL} ( 2, i )$-action for each $i \in \tI$ as $R$ and $\bX'$ do. It follows that
$$E \subset \bigcap _{w \in W} \{ v_{w \la}^* = 0 \}.$$
For each $\la \in P_+$, we have a natural map
$$\psi_{\la} : \bX \rightarrow \mathbb P ( H ^0 ( \bX, \cO_{\bX} ( \la ) )^{*} ) = \mathbb P ( L ( \la ) ^{\wedge} )$$
by Theorem \ref{tBW}.

\begin{claim}
The map $\psi_{\la}$ sends $E$ to $\mathbb P ( M^{\wedge} )$, where $M \subset L ( \la )$ is a $U ( \g )$-stable $H$-submodule that does not contain $H$-weight $\{w \la \}_{w \in W}$-part for each $\la \in P_+$.
\end{claim}

\begin{proof}
Assume to the contrary to deduce contradiction. Then, we have some $x \in E$ so that $\psi_{\la} ( x ) \not\in \mathbb P ( M^{\wedge} )$ for every $U ( \g )$-stable $H$-submodule that does not contain $H$-weight $\{w \la \}_{w \in W}$-part. Then, applying $\mathop{SL} ( 2, i )$-action repeatedly, we obtain a point $y \in E$ so that $\psi_{\la} ( y ) \in \{ v_{\la}^* \neq 0 \}$. This is a contradiction and we conclude the result.
\end{proof}

We return to the proof of Theorem \ref{id}. By taking the fixed point of a $\Gm$-action that shrinks $\widehat{N}^-$, we deduce that
$$E^{H} \cap \bX^H = \emptyset.$$
This forces $E = \emptyset$ (our $\Gm$-action always send a point to a limit point as the set of $H$-weight of $L ( \la )$ in contained in $\la - \Z_{\ge 0}\Delta^+$), and we conclude the assertion.
\end{proof}

\begin{cor}[of the proof of Theorem \ref{id}]\label{union}
We have $\bX = \bigsqcup _{w \in W} \mathbb O ^w$. \hfill $\Box$
\end{cor}

\begin{cor}
We have $X_w = \overline{\mathbb O _w}$, and the thin flag variety $X$ of $\g$ is obtained as $\bigcup_{w \in W} X_w$ inside $\bX$. \hfill $\Box$
\end{cor}

\begin{thm}[Kashiwara \cite{Kas89} \S 4 and Kashiwara-Tanisaki \cite{KT95} \S 1.3]\label{closure-rel}
For each $w, v \in W$, we have:
\begin{enumerate}
\item $\bO_w \subset X_v$ if and only if $w \le v$;
\item $\bO^w \subset \bX^v$ if and only if $w \ge v$.
\end{enumerate}
Moreover, we have $\dim \, X_w = \ell ( w )$ and $\mathrm{codim} _{\bX} \, \bX^w = \ell ( w )$.
\end{thm}

\section{Frobenius splitting of thick flag manifolds}

We retain the setting of the previous section. Let $B := N^+ H \subset \hB^+$. For each $i \in \tI$, we have an overgroup $B \subset B_i \subset \hB^+_i$ so that $\mathrm{Lie} \, B_i \cong \Bbbk F_i \oplus \mathrm{Lie} \, B$. We similarly define $B^- := N^- H$ and $B_i^-$ for each $i \in \tI$. Let $\bi = (i_{1},i_{2},\ldots,i_{\ell}) \in \tI^{\ell}$ be a sequence. We have a Bott-Samelson-Demazure-Hansen variety
$$Z ( \bi ) := B_{i_{1}} \times^{B} B_{i_{2}} \times^{B} \cdots \times^{B} B_{i_{\ell}} / B.$$
In case $w = s_{i_1} s_{i_2} \cdots s_{i_{\ell}}$ satisfies $\ell ( w ) = \ell$ (i.e. $\bi$ is a reduced expression of $w$), we have the BSDH resolution (see e.g. \cite[Chapter V\!I\!I\!I]{Kum02})
$$\pi_{\bi} : Z ( \bi ) \ni ( g_{1}, g_{2}, \ldots, g_{\ell} ) \mapsto g_{1} g_{2} \cdots g_{\ell} B / B \in X_{w}.$$
The variety $Z ( \bi )$ admits a left $B$-action, that makes $\pi_{\bi}$ into a $B$-equivariant morphism. For each $1 \le k \le \ell$, we define a $B$-stable divisor $H_{k} \subset Z (\bi)$ by requiring $g_{k} \in B$ for $(g_{1},g_{2},\ldots,g_{\ell}) \in Z ( \bi )$. Note that $H_k$ is naturally isomorphic to $Z (\bi^k)$, where $\bi^k \in \tI^{\ell - 1}$ is obtained from $\bi$ by omitting the $k$-th entry. In addition, every subword $\bi' = (i_{j_1}, \ldots, i_{j_\ell'}) \in \tI^{\ell'}$ of $\bi$ (so that $1 \le j_1 < j_2 < \cdots < j_{\ell'} < \ell$) gives us a $B$-equivariant embedding described as
$$Z ( \bi' ) \ni (g_1,\ldots,g_{\ell'}) \mapsto (\overbrace{1,\ldots,1}^{j_1 -1},g_{j_1}, \overbrace{1 \ldots, 1}^{j_2 - j_1 - 1}, g_{j_2}, \ldots ) \in Z ( \bi ).$$

We follow the generality on Frobenius splitting in \cite{BK05}, that considers separated schemes of finite type. We sometimes use the assertions from \cite{BK05} without finite type assumption when the assertion is independent of that, whose typical disguises are properness, finite generation, and the Serre vanishing theorem. Note that a closed subscheme of a projective space is separated.

\begin{defn}[Frobenius splitting of a ring]
Let $R$ be a commutative ring over $\Bbbk$ with characteristic $p > 0$, and let $R^{(1)}$ denote the set $R$ equipped with the map
$$R \times R^{(1)} \ni (r,m) \mapsto r^p m \in R^{(1)}.$$
This equips $R^{(1)}$ an $R$-module structure over $\Bbbk$ (the $\Bbbk$-vector space structure on $R^{(1)}$ is also twisted by the $p$-th power operation), together with an inclusion $\imath : R . 1 \subset R^{(1)}$. An $R$-module map $\phi : R^{(1)} \to R$ is said to be a Frobenius splitting if $\phi \circ \imath$ is an identity.
\end{defn}

\begin{defn}[Frobenius splitting of a scheme]
Let $\mathfrak X$ be a separated scheme defined over a field $\Bbbk$ with positive characteristic. Let $\Fr$ be the (relative) Frobenius endomorphism of $\mathfrak X$ (that induces a $\Bbbk$-linear endomorphism). We have a natural inclusion $\imath : \cO_{\mathfrak X} \rightarrow \Fr_{*} \cO_{\mathfrak X}$. A Frobenius splitting of $\mathfrak X$ is a $\cO_{\mathfrak X}$-linear morphism $\phi : \Fr_{*} \cO_{\mathfrak X} \rightarrow \cO_{\mathfrak X}$ so that the composition $\phi \circ \imath$ is the identity.
\end{defn}

\begin{defn}[Compatible splitting]
Let $\gY \subset \gX$ be an inclusion of separated schemes defined over $\Bbbk$. A Frobenius splitting $\phi$ of $\gX$ is said to be compatible with $\gY$ if $\phi (\mathsf{Fr}_* \mathcal I _{\gY} ) \subset \mathcal I_{\gY}$.
\end{defn}

\begin{rem}
A Frobenius splitting of $\gX$ compatible with $\gY$ induces a Frobenius splitting of $\gY$ (see e.g. \cite[Remark 1.1.4 (ii)]{BK05}).
\end{rem}

\begin{thm}[\cite{BK05} Lemma 1.1.11 and Exercise 1.1.E]\label{F-rel}
Let $\mathfrak X$ be a separated scheme of finite type over $\Bbbk$ with semiample line bundles $\mathcal L_1,\ldots, \mathcal L_r$. If $\mathfrak X$ admits a Frobenius splitting, then the multi-section ring
$$\bigoplus_{n_1,\ldots,n_r \ge 0} \Gamma ( \mathfrak X, \mathcal L_1 ^{\otimes n_1} \otimes \cdots \otimes \mathcal L_r ^{\otimes n_r} )$$
admits a Frobenius splitting $\phi$. Moreover, a closed subscheme $\gY \subset \gX = \mathrm{Proj} \, S$ admits a compatible Frobenius splitting if and only if the homogeneous ideal $I _{\gY} \subset S$ that defines $\gY$ satisfies $\phi ( I_{\gY} ) \subset I_{\gY}$. \hfill $\Box$
\end{thm}

\begin{defn}[$B$-canonical splitting]
Let $\gX$ be a separated scheme equipped with a $B$-action. A Frobenius splitting $\phi$ is said to be $B$-canonical if it is $H$-fixed, and each $i \in \mathtt I$ yields
\begin{equation}
\rho_{\al_i} ( z ) \phi ( \rho_{\al_i} ( - z ) f )  = \sum_{j = 0}^{p-1} \phi_{i, j} ( f ),\label{Bcaneq}
\end{equation}
where $\phi_{i, j} \in \Hom_{\cO_{\mathfrak X}} ( \Fr_{*} \cO_{\mathfrak X}, \cO_{\mathfrak X} )$. We similarly define the notion of $B^-$-canonical splitting by using $\{ \rho_{-\al_i} \}_{i \in \tI}$ instead. The $B$-canonical splitting of a commutative ring $S$ over $\Bbbk$ is defined through its spectrum.
\end{defn}

\begin{thm}[\cite{BK05} Exercise 4.1.E.2]\label{uniq}
Assume that $\mathrm{char} \, \Bbbk > 0$. For each $\bi \in \tI^\ell$, there exists a unique $B$-canonical splitting of $Z ( \bi )$ that is compatible with the subvarieties $Z (\bi')$ obtained by subwords $\bi'$ of $\bi$. \hfill $\Box$
\end{thm}

\begin{cor}\label{rest}
In the setting of Theorem \ref{uniq}, the restriction of the $B$-canonical splitting to $Z (\bi')$ is $B$-canonical.
\end{cor}

\begin{proof}
The condition of $B$-canonical splitting is preserved by the restriction to a $B$-stable compatibly split subset.
\end{proof}

\begin{lem}\label{Z-dense2}
For each $w \in W$, the ind-scheme $( X \cap \bX^w )$ is Zariski dense in $\bX^w$.
\end{lem}

\begin{proof}
Assume to the contrary to deduce contradiction. Let $w = s_{i_1} s_{i_2} \cdots s_{i_{\ell}}$ be a reduced expression. For a $B^-$-stable subset $Y \subset \bX^v$ that is not Zariski dense in $\bX^v$ and $i \in \tI$ so that $s_i v < v$, the inclusion
$$\mathop{SL} ( 2, i ) Y \subset \overline{\rho_{\al_i} ( \mathbb G_a ) Y} \subset \mathop{SL} ( 2, i ) \bX^{v} = \bX^{v} \cup \bX^{s_i v} = \bX^{s_i v}$$
cannot be Zariski dense. Moreover, $\mathop{SL} ( 2, i ) Y$ is again $B^-$-stable by the Bruhat decomposition (of $\mathop{SL} ( 2, i )$). As $( X \cap \bX^w )$ is stable under the action of $B^-$, we repeatedly apply the above estimate to conclude
$$\overline{\mathop{SL} ( 2, i_{\ell} ) \cdots \mathop{SL} ( 2, i_{1} ) ( X \cap \bX ^w )} \subset \bX$$
is not Zariski dense. By the Bruhat decomposition, we have $\bX ^{s_i v} \subset \mathop{SL} ( 2, i ) \bX ^v$ for each $i \in \tI$ and $v \in W$. Each rational point $x$ of $\bX$ satisfies
$$\mathop{SL} ( 2, i_1 ) \cdots \mathop{SL} ( 2, i_{\ell} ) x \cap \bX^w \neq \emptyset$$
by its repeated application. It follows that
$$X = \mathop{SL} ( 2, i_{\ell} ) \cdots \mathop{SL} ( 2, i_1 ) ( X \cap \bX ^w ) \subset \bX$$
is also not Zariski dense. This gives a contradiction to Lemma \ref{Z-dense}, and we conclude the result.
\end{proof}

\begin{lem}\label{B-canonical}
Assume that $\mathrm{char} \, \Bbbk > 0$. For each $w \in W$, the ring $R$ and $R_w$ admits a $B$-canonical splitting.
\end{lem}

\begin{proof}
Let $\bi \in \tI^{\ell}$ be a sequence so that $\mathrm{Im} \, \pi_{\bi} = X_w$. Then, we have
$$H^0 ( Z ( \bi ), \pi_{\bi}^* \cO_{X} ( \la ) ) \cong H^0 ( X_w, \cO_X ( \la ) ) \cong L _w ( \la )^* \hskip 3mm \la \in P_+$$
by \cite[Th\'eor\`eme 3]{Mat88} (cf. \cite[Theorem 8.2.2]{Kum02}).

Applying Theorem \ref{F-rel}, we deduce that the ring $R_w$ admits a $B$-canonical splitting. Choose a series of sequences $\bi_k \in I^k$ ($k \ge 1$) so that
\begin{enumerate}
\item $\bi_{k}$ is obtained from $\bi_{k+1}$ by omitting the first entry:
\item $\bigcup_{k \ge 1} \pi_{\bi_k} ( Z ( \bi_k ) ) = X$,
\end{enumerate}
(whose existence is guaranteed by the subword property of the Bruhat order \cite[Lemma 1.3.16]{Kum02}). Let $w _k \in W$ be so that $X_{w_k} = \pi_{\bi_k} ( Z ( \bi_k ) )$ (that exists as $Z ( \bi_k )$ is irreducible). Then, we have
$$L ( \la ) = \varinjlim _k L _{w_k} ( \la ).$$
This induces a dense inclusion of algebras
$$R \subset \varprojlim _k R_{w_k},$$
where the LHS is the $H$-finite part of the RHS. The system $\{ R_{w_k} \}_{k \ge 1}$ is an inverse system with surjective transition maps. Therefore, Corollary \ref{rest} induces a Frobenius splitting of $\varprojlim _k R_{w_k}$ from the $B$-canonical splittings of $\{ R_{w_k} \}_{k \ge 1}$. Since our splitting preserves the $H$-weights, it descends to the $H$-finite part $R$ as required.
\end{proof}

\begin{cor}\label{B-can1}
Assume that $\mathrm{char} \, \Bbbk > 0$. The ring $R$ admits a $B^-$-canonical splitting, and hence $\bX$ is $B^-$-canonically Frobenius split.
\end{cor}

\begin{proof}
We retain the setting of the proof of Lemma \ref{B-canonical}. Our ring $R$ is a $H$-finite graded algebra that admits a $B$-canonical splitting. Note that $R$ admits a rational action of $\mathop{SL} ( 2, i )$ for each $i \in \tI$ as each $L ( \la )$ is integrable. Hence, \cite[Excercise 4.1 (1)]{BK05} forces a $B$-canonical splitting of $R$ to induce a $B^-$-canonical splitting as desired.
\end{proof}

\begin{cor}\label{B-can-}
Assume that $\mathrm{char} \, \Bbbk > 0$. For each $w \in W$, the $B^-$-canonical splitting of $\bX$ $($constructed above$)$ is compatible with $\mathbb X^w$.
\end{cor}

\begin{proof}
We argue along the line of \cite[Proposition 5.3]{KS14}, that was stated with the symmetrizability assumption (that we drop here).

We already know that the scheme $\mathbb X$ (or rather its projective coordinate ring) admits a $B^-$-canonical splitting by Corollary \ref{B-can1}.

We show that our splitting splits the $H$-fixed points as in \cite[Proof of Proposition 5.3 Assertion I\!I]{KS14}. The $H$-fixed point $x_w$ of $\bX$ corresponding to $w \in W$ is contained in $X_w$. Hence, we have $H$-algebra morphisms
$$\bigoplus_{\la \in P_+} \Bbbk_{-w \la} \hookrightarrow R_w \rightarrow \bigoplus_{\la \in P_+} \Bbbk_{-w \la} $$
corresponding to $x_w \in X_w$, whose composition is the identity. As our Frobenius splitting induces that of $R_w$ and preserves $H$-weight spaces, we conclude that our splitting splits the $H$-fixed points of $\bX$ by Lemma \ref{fixed}.

We show that our splitting splits each $\bX^w$ compatibly as in \cite[Proof of Proposition 5.3 Assertion I\!I\!I]{KS14} to complete the proof. Let $I_w$ be the ideal of $R$ corresponding to $x_w$. The ideal $I_w$ is preserved by our Frobenius splitting. Therefore, the ideal $I^w := \cap_{b \in B^-} b \cdot I_w \subset R$ is preserved by our $B^-$-canonical splitting thanks to \cite[Proposition 4.1.8]{BK05}. By Lemma \ref{Z-dense2}, the ideal $I^w$ defines the Zariski closure of $\hB^- x_w$ (as that is the same as $B^- x_w$) inside $\bX$, that is $\bX^w$. It follows that $\bX$ splits compatibly with $\bX^w$ through our splitting as required.
\end{proof}

\begin{rem}
According to Kumar-Schwede \cite{KS14}, the essential part of our proof of Corollary \ref{B-can-} traces back to a result of Olivier Mathieu. As the author has no access to it, he cites it from \cite{KS14}.
\end{rem}

\begin{cor}
For each $w \in W$, the scheme $\mathbb X^w$ is integral.
\end{cor}

\begin{proof}
Apply \cite[Proposition 1.2.1]{BK05} to Corollary \ref{B-can-} if $\mathrm{char} \, \Bbbk > 0$. As the integrality of $\mathbb X^w$ follows by the integrality of $R^w$, we apply \cite[Proposition 1.6.5]{BK05} to subalgebras of $R^w$ generated by finitely many $H$-weight spaces (so that it is finitely generated) to deduce the integrality in $\mathrm{char} \, \Bbbk = 0$.
\end{proof}

By restricting $\cO_{\bX} ( \la )$ ($\la \in P$), we obtain a line bundle $\cO_{\bX^w} ( \la )$ on $\bX^w$ for each $w \in W$.

Let $\tJ \subset \tI$. Consider the subring
$$R_{\tJ} := \bigoplus _{\la \in P_{+}^{\tJ}} L ( \la )^{\vee} \subset R.$$
We set $\bX_{\tJ} := \mathrm{Proj} \, R_{\tJ}$. This also defines a line bundle $\cO_{\bX_{\tJ}} ( \la )$ for each $\tJ \subset \tI$ and $\la \in P_+^{\tJ}$. We have natural map
$$\pi_{\tJ} : \bX \longrightarrow \bX_{\tJ}.$$

\begin{lem}\label{par}
Let $\tJ \subset \tI$. The morphism $\pi_{\tJ}$ is $G^-$-equivariant and surjective. We have a $B^-$-canonical splitting of $\bX_{\tJ}$ that is compatible with the $\hB^{-}$-orbits.
\end{lem}

\begin{proof}
Since the dual of the homogeneous coordinate rings of $\bX$ and $\bX_{\tJ}$ admits the $\widehat{B}^-$-action and $N ( H )$-action, we conclude that $\pi_{\tJ}$ is equivariant with respect to the group generated by $\widehat{B}^-$ and $N ( H )$, that is $G^-$.

The $B^{-}$-canonical splitting of $\bX$ induces that of $\bX_{\tJ}$ through the description of its projective coordinate ring. This must be compatible with the Zariski closure of the image of $\hB^{-}$-orbits. Hence, it remains to show that $\pi_{\tJ}$ is surjective.

Fix $w \in W$. The analogous map to $\pi_{\tJ}$ defined for $X_{w}$ is surjective (see \cite[Proposition 7.1.15]{Kum02}). The same proof as Lemma \ref{fixed} (relying on \cite{Kum02}) implies $\bX_{\tJ}^{H} \subset \pi_{\tJ} ( \bX^{H} )$. Hence, the same argument as in Theorem \ref{id} yields that every $\hB^{-}$-orbit of $\bX_{\tJ}$ is the image of a $\hB^{-}$-orbit of $\bX$ as required.
\end{proof}

\begin{lem}\label{push}
Let $\tJ \subset \tI$. The fiber of $\pi_{\tJ}$ is isomorphic to the thick flag manifold of the Kac-Moody subalgebra of $\g$ corresponding to $\tJ$. Moreover, we have $( \pi_{\tJ} )_* \cO_{\bX} ( \la ) \cong \cO_{\bX_{\tJ}} ( \la )$ for $\la \in P_+^{\tJ}$.
\end{lem}

\begin{proof}
Let $\g'$ denote the Kac-Moody algebra that is a subalgebra of $\g$ corresponding to $\tJ$, and let $W'$ denote its Weyl group that is a subgroup of $W$. Let $R^{\tJ}$ be the minimal homogeneous coordinate ring of $\pi_{\tJ}^{-1} ( B_{\tJ} / B_{\tJ} )$ so that we have an algebra map $\phi : R \rightarrow R^{\tJ}$ corresponding to $\pi_{\tJ}^{-1} ( B_{\tJ} / B_{\tJ} ) \subset \bX$. Since $\bX_\tJ$ is $G^-$-homogeneous, we find that the scheme $\pi_{\tJ}^{-1} ( B_{\tJ} / B_{\tJ} )$ is reduced.

Let $\hN_{\tJ}^- \subset \hB^-$ be the pro-unipotent radical of $\hB^-_{\tJ}$. We find a $H$-stable complementary pro-unipotent group $\widehat{U} \subset \hB^-$ so that $\hN^- = \widehat{U} \hN_{\tJ}^-$, $\widehat{U}$ normalizes $\hN_{\tJ}^-$, and $\widehat{U} \cap \hN_{\tJ}^- = \{ \mathrm{id} \}$.

A point $x \in \bX$ is written as $x = g \dot{w} B$ for some $g \in \hN$ and a lift $\dot{w} \in N ( H )$ of $w \in W$, that gives a point $[g \dot{w} v_{\varpi_i}] \in \P ( L ( \varpi_i )^{\wedge})$ for each $i \in \tI$. If $g \not\in \widehat{U}$, then we have
$$g \dot{w} v_{\varpi_i} \in \Bbbk g v_{\varpi_i} \not\subset \Bbbk v_{\varpi_i} \hskip 5mm \text{for some} \hskip 2mm i \not\in \tJ \text{ and every } w \in W'.$$
Note that $w \in W$ belongs to $W'$ if and only if $w \varpi_i = \varpi _i$ for every $i \not\in \tJ$. Therefore, we find that every point in $\pi_{\tJ}^{-1} ( B_{\tJ} / B_{\tJ} )$ is of the form $x = g \dot{w} B$ for $g \in \widehat{U}$ and $w \in W'$ (by Corollary \ref{union} and Lemma \ref{par}).

Therefore, if we represent a point $x \in \pi_{\tJ}^{-1} ( B_{\tJ} / B_{\tJ} )$ as a point $( [x_i] ) \in \prod_{i \in \tI} \P ( L ( \varpi_i )^{\wedge} )$ (using Definition \ref{mproj}), then the vector $x_i$ does not contain $H$-weights except for $\varpi_i - \Z_{\ge 0} \{ \al_i \mid i \in \tJ \}$. By the cocyclicity of the dual Verma modules, it means that $x_i$ belongs to the (maximal) integrable highest weight module $L ' ( \la )$ of $\g'$ spanned by $v_{\varpi_i}$. Moreover, the $H$-weight comparison implies that $L ' ( \la )^{\vee} \subset L ( \la ) ^{\vee}$ is precisely the $\hN^-_{\tJ}$-invariant part.

Since $\pi_{\tJ}^{-1} ( B_{\tJ} / B_{\tJ} )$ is $\hN^-_{\tJ}$-invariant and reduced, it follows that the map $\phi$ factors through
$$R^{\tJ} := \bigoplus_{i \in \tJ} L ' ( \varpi_i )^{\vee},$$
that is the homogeneous coordinate ring of the thick flag manifold of $\g'$. As $\pi_{\tJ}^{-1} ( B_{\tJ} / B_{\tJ} )$ is a closed subscheme of $\bX$, we conclude the $\phi$ must be in fact an equality. Hence, the fibers of $\pi_{\tJ}$ are isomorphic to the thick flag manifold of $\g'$.

By examining the sections on the fibers of $\pi_{\tJ}$, we conclude that $( \pi_{\tJ} )_* \cO_{\bX} ( \la )$ is a line bundle. Since $\bX$ is homogeneous and $( \pi_{\tJ} )_* \cO_{\bX} ( \la )$ is $G^-$-equivariant, we conclude the assertion by the comparison (of characters) on fibers.
\end{proof}

\begin{thm}[\cite{KS09} second part of Conjecture 8.10]\label{rest-surj}
For each $\la \in P_+$ and $w, v \in W$ so that $v < w$, the natural restriction map
$$H^0 ( \bX^v, \cO_{\bX^v} (\la) ) \longrightarrow H^0 ( \bX^w, \cO_{\bX^w} (\la) )$$
is surjective.
\end{thm}

\begin{proof}
We set $\tJ := \{i \in \tI \mid \left< \al^{\vee}_i, \la \right> = 0 \}$. By Lemma \ref{push}, the assertion reduces to the surjectivity of
$$H^0 ( \bX^v_{\tJ}, \cO_{\bX^v_{\tJ}} (\la) ) \longrightarrow H^0 ( \bX^w_{\tJ}, \cO_{\bX^w_{\tJ}} (\la) ),$$
where $\bX^w_{\tJ} := \pi_{\tJ} ( \bX ^{w} )$ for each $w \in W$. We might omit the subscript $\tJ$ in the below for simplicity. Note that $\cO_{\bX_{\tJ}} (\la) = \cO_{\bX} (\la)$ is ample. By the associativity of the restriction maps (and Theorem \ref{closure-rel}), we can assume $v = e$.

For each $v \in W$ and $w \in W$, we have a restriction map
$$\varphi^v_w : H^0 ( \bX^v, \cO_{\bX^v} (\la) ) \longrightarrow H^0 ( \bX^v \cap X_w, \cO_{\bX^v \cap X_w} (\la) ).$$
By Lemma \ref{Z-dense2}, the inverse limit of $\{ \varphi^v_w\}_w$ yields an inclusion
\begin{equation}
\varphi^v : H^0 ( \bX^v, \cO_{\bX^v} (\la) ) \hookrightarrow H^0 ( \bX^v \cap X, \cO_{\bX^v \cap X} (\la) ).\label{thick-to-thin}
\end{equation}

Let $\Psi \subset \Delta^+$ be a finite set. Let us consider a linear functional $h$ on $\Delta^+ \subset X ^* ( H ) \otimes_\Z \R$ so that $0 < h ( \alpha_i )$ for each $i \in \tI$ and $h ( \Psi ) < 1$. Then, the subset
$$\Delta^+ ( h ) := \{ \beta \in \Delta^+ \mid h ( \beta ) < 1 \} \subset \Delta^+$$
is finite, and every $\Z_{\ge 0}$-linear combination of elements of $\Delta^+ \backslash \Delta^+ ( h )$ does not belong to $\Delta^+ ( h )$.

For each $w \in W$, the set of $H$-weights of $\bigoplus _{i \in \tI} L_w ( \varpi_i )$ is finite, and hence so is the set $\Psi_w$ of positive roots obtained by the difference of two $H$-weights of $\bigoplus _{i \in \tI} L_w ( \varpi_i )$. Applying the above construction, we can find a partition $\Delta^+ = \Delta^+_1 \sqcup \Delta^+_2$ ($\Delta^+_1$ is $\Delta^+ ( h )$ obtained by setting $\Psi = \Psi_w$) so that every $x \in \hN^-$ factors into $x = x_1 x_2$, where $x _2$ is the product of one-parameter subgroup corresponding to $\Delta^+_2$, and
$$L _w ( \varpi_i ) \cap x_2 L _w ( \varpi_i ) = \{ v \in L _w ( \varpi_i ) \mid x_2 v = v \}.$$ 
This implies
$$\bX^v \cap X_w = \overline{\hB^- x_v} \cap \overline{B x_w} = \overline{B^- x_v} \cap \overline{B x_w},$$
where the most RHS is the definition of the Richardson variety in \cite{KS14} (when $\tJ = \emptyset$).

In case $\tJ = \emptyset$, \cite[Proposition 5.3]{KS14} equips $( \bX^v \cap X_{w} )$ a Frobenius splitting compatible with $( \bX^{v'} \cap X_{w'} )$'s in its closure.

In case $\tJ \neq \emptyset$, the pullback of $( \bX^v_{\tJ} \cap X_{w, \tJ} )$ to $\bX$ is a (possible infinite) union of Richardson varieties of $X \subset \bX$. Therefore, we can transplant the Frobenius splitting $\phi$ (that we have constructed through the Richardson varieties) on $\bX$ to a Frobenius splitting of $( \bX^v_{\tJ} \cap X_{w, \tJ} )$ compatible with $( \bX^{v'}_{\tJ} \cap X_{w', \tJ} )$'s in its closure through
$$\mathsf{Fr} _* \cO _{\bX_{\tJ}} \to ( \pi_{\tJ} ) _* \mathsf{Fr} _* \cO _{\bX} \stackrel{\pi_* \phi}{\longrightarrow} ( \pi_{\tJ} ) _* \cO _{\bX} \cong \cO_{\bX_{\tJ}}.$$

Therefore, \cite[Theorem 1.2.8]{BK05} yields that the map
$$H^0 ( \bX^v \cap X_w, \cO_{\bX^v \cap X_w} (\la) ) \longrightarrow \!\!\!\!\! \rightarrow H^0 ( \bX^{v'} \cap X_w, \cO_{\bX^{v'} \cap X_w} (\la) )$$
is surjective for every $w, v, v' \in W$ so that $v \le v'$ when $\mathrm{char} \, \Bbbk > 0$. Since the both of $(\bX^v \cap X_w)$ and $(\bX^{v'} \cap X_w)$ are finite type schemes, \cite[Corollary 1.6.3]{BK05} lifts this surjection to the case of $\mathrm{char} \, \Bbbk = 0$. Hence, we deduce a surjection
$$H^0 ( \bX^v \cap X, \cO_{\bX^v \cap X} (\la) ) \longrightarrow \!\!\!\!\! \rightarrow H^0 ( \bX^{v'} \cap X, \cO_{\bX^{v'} \cap X} (\la) )$$
for every $v, v' \in W$ so that $v \le v'$ by taking the inverse limits with respect to surjective inverse systems (so that they satisfies the Mittag-Leffler condition), regardless of the characteristic.

The space $H^0 ( \bX^v, \cO_{\bX^v} (\la) )$ is $H$-finite since
$$H^0 ( \bX^v, \cO_{\bX^v} (\la) ) \subset H^0 ( \mathbb O^v, \cO_{\mathbb O^v} (\la) ) \cong \Bbbk [\mathbb O^v] \otimes_{\Bbbk} \Bbbk _{-\la}.$$

In case $v = e$, the LHS of (\ref{thick-to-thin}) is given in Theorem \ref{tBW}, and the RHS is given in \cite{Mat88} (cf. Theorem \ref{thin-coh}). In particular, the LHS is the $H$-finite part of the RHS.
Therefore, the commutative diagram
\begin{equation}
\xymatrix{
H^0 ( \bX, \cO_{\bX} (\la) ) \ar@{^{(}->}[r] \ar[d] & H^0 ( \bX \cap X, \cO_{\bX \cap X} (\la) ) \ar@{->>}[d]\\
H^0 ( \bX^v, \cO_{\bX^v} (\la) ) \ar@{^{(}->}[r] & H^0 ( \bX^{v} \cap X, \cO_{\bX^{v} \cap X} (\la) )
}\label{4term}
\end{equation}
yields the surjectivity of the left vertical arrow, that implies our assertion.
\end{proof}

\begin{cor}\label{grest-surj}
Let $\mathbb Y, \mathbb Y'$ be reduced unions of thick Schubert varieties so that $\mathbb Y' \subset \mathbb Y$. Then, the natural restriction map
$$H^0 ( \mathbb Y, \cO_{\mathbb Y} (\la) ) \longrightarrow H^0 ( \mathbb Y', \cO_{\mathbb Y'} (\la) )$$
is surjective for each $\la \in P_+$.
\end{cor}

\begin{proof}
We can formally replace $\bX^v \cap X_w$ with reduced unions of them (that are compatibly split along the intersections by our canonical splitting) in the proof of Theorem \ref{rest-surj} to deduce the assertion. 
\end{proof}

The following is a consequence of Theorem \ref{rest-surj} as described in \cite[\S 8]{KS09} after Conjecture 8.10 (when $\g$ is of affine type).

\begin{thm}\label{Gamma}
For each $\la \in P_+$, we have
$$H^0 ( \bX^w, \cO_{\bX^w} (\la) ) \cong L^w ( \la )^{\vee}.$$
\end{thm}

\begin{proof}
Combining Theorem \ref{tBW} and Theorem \ref{rest-surj}, we have
$$H^0 ( \bX^w, \cO_{\bX^w} (\la) ) ^{\vee} \subset L ( \la ).$$
In addition, the integrality of $\bX^w$ implies that $H^0 ( \bX^w, \cO_{\bX^w} (\la) ) ^{\vee}$ is cyclic as its covering module $H^0 ( \bO^w, \cO_{\bX^w} (\la) ) ^{\vee}$ is a $U ( \gn^- )$-module with cyclic $H$-eigenvector $v_{w \la}$ by \cite[Lemme 4]{Mat89}. These imply our result.
\end{proof}

\begin{thm}[Kashiwara-Shimozono \cite{KS09} Proposition 3.2]\label{tsn}
For each $w \in W$, the scheme $\bX^w$ is normal.
\end{thm}

\begin{proof}
The argument in \cite[Proposition 3.2]{KS09} is stated for symmetrizable $\g$ and $\mathrm{char} \, \Bbbk = 0$, but there are no place this assumption is used until \cite[Proposition 3.2]{KS09} in the main body of \cite{KS09}.
\end{proof}

\begin{cor}\label{pn}
For each $w \in W$, we have an isomorphism
$$\bX^w \cong \mathrm{Proj} \, R^w.$$
In particular, $\bX^w$ is projectively normal.
\end{cor}

\begin{proof}
The first assertion is the direct consequence of Theorem \ref{Gamma} since $\bX^w$ is a closed subscheme of $\bX$.

We prove the second assertion. Corollary \ref{tmult} and Lemma \ref{incl} asserts that the ring $R^w$ is generated by $\bigoplus_{i \in \tI} L ^w ( \varpi_i )^{\vee}$. This verifies a sufficient condition of projective normality (see e.g. Hartshorne \cite[Chapter I\!I, Exercise 5.14]{Har77} for singly graded case) in the presence of the normality of $\bX^w$. Therefore, we conclude the assertion.
\end{proof}

The following result implies that $\{L^w ( \la )\}_{w \in W}$ forms a filtration of $L ( \la )$ for each $\la \in P_+$, that is previously recorded when $\g$ is of affine type (see \cite[Theorem 6.23]{AKT08}). An analogous result is known for $\{ L_w ( \la )\}_{w \in W}$ by the works of many people (cf. Littelmann \cite[\S 8]{Lit98} and Kumar \cite[V\!I\!I\!I]{Kum02}).

\begin{cor}\label{compat}
For each finite subset $S \subset W$, there exists another subset $S' \subset W$ so that
$$\bigcap_{w \in S} L^w ( \la ) = \sum_{v \in S'} L^v ( \la ).$$
\end{cor}

\begin{proof}
Let $T \subset W$ and set $\bX ( T ) := \bigcup _{w \in T} \bX^w$ (here the union is understood to be the reduced union). We have a sequence of maps
$$\cO_{\bX} ( \la ) \longrightarrow \!\!\!\!\! \rightarrow \cO_{\bX ( T )} ( \la ) \hookrightarrow \bigoplus_{w \in T} \cO_{\bX^w} ( \la ).$$
Thanks to Corollary \ref{grest-surj}, we deduce
$$\bigoplus_{w \in T} L^w ( \la ) = \bigoplus_{w \in T} \Gamma ( \bX, \cO_{\bX^w} ( \la ) ) ^{\vee} \longrightarrow \!\!\!\!\! \rightarrow \Gamma ( \bX, \cO_{\bX ( T )} ( \la ) ) ^{\vee} \hookrightarrow \Gamma ( \bX, \cO_{\bX} ( \la ) ) ^{\vee} = L ( \la ).$$
Moreover, the restriction of the composition maps to a direct summand $L^w ( \la )$ yields the standard embedding. Thus, we conclude
\begin{equation}
\Gamma ( \bX, \cO_{\bX ( T )} ( \la ) ) ^{\vee} = \sum_{w \in T} L^w ( \la ) \subset L ( \la ).\label{Gamma-desc}
\end{equation}

Let us divide $T = T_1 \sqcup T_2$, and we set $\mathbb Y_i := \bigcup_{w \in T_i} \bX^w$ for $i = 1,2$. By \cite[Proposition 1.2.1]{BK05}, the scheme $\mathbb Y_i$ ($i=1,2$) and the scheme-theoretic intersection $\mathbb Y := \mathbb Y_1 \cap \mathbb Y_2$ are reduced.

Since $\mathbb Y$ is $\hB^-$-stable, we have $\mathbb Y = \bX ( T' )$ for some $T' \subset W$. We have a short exact sequence
$$0 \rightarrow \cO_{\bX ( T' )} ( \la ) \rightarrow \cO_{\mathbb Y_1} ( \la ) \oplus \cO_{\mathbb Y_2} ( \la ) \rightarrow \cO_{\mathbb Y_1 \cup \mathbb Y_2} ( \la ) \rightarrow 0.$$
Thanks to (\ref{Gamma-desc}) and Corollary \ref{grest-surj}, we conclude a short exact sequence
$$0 \rightarrow \sum_{w \in T} L^w ( \la ) \rightarrow \left( \sum_{w \in T_1} L^w ( \la ) \right) \oplus \left( \sum_{w \in T_2} L^w ( \la ) \right) \rightarrow \sum_{w \in T'} L^w ( \la ) \rightarrow 0$$
of $\gb^-$-modules. In particular, the third term can be identified with the intersection of the direct summands of the second term inside $L ( \lambda )$. This proves the assertion by induction on $|S|$ (since the case $|S| = 1$ is apparent from $S = S'$).
\end{proof}

{\small
\hskip -5.25mm {\bf Acknowledgement:} The author would like to thank Masaki Kashiwara and Shrawan Kumar for helpful correspondences, and Daisuke Sagaki for discussions. This research is supported in part by JSPS Grant-in-Aid for Scientific Research (B) 26287004.}

{\footnotesize
\bibliography{kmref}

\begin{thebibliography}{10}

\bibitem{AKT08}
Susumu Ariki, Victor Kreiman, and Shunsuke Tsuchioka.
\newblock On the tensor product of two basic representations of
  {$U_v(\widehat{\mathfrak{sl}}_e)$}.
\newblock {\em Adv. Math.}, 218(1):28--86, 2008.

\bibitem{BK05}
Michel Brion and Shrawan Kumar.
\newblock {\em Frobenius splitting methods in geometry and representation
  theory}, volume 231 of {\em Progress in Mathematics}.
\newblock Birkh\"auser Boston, Inc., Boston, MA, 2005.

\bibitem{FGT}
Susanna Fishel, Ian Grojnowski, and Constantin Teleman.
\newblock The strong {M}acdonald conjecture and {H}odge theory on the loop
  {G}rassmannian.
\newblock {\em Ann. of Math. (2)}, 168(1):175--220, 2008.

\bibitem{EGAI}
A.~Grothendieck.
\newblock \'{E}l\'ements de g\'eom\'etrie alg\'ebrique. {I}. {L}e langage des
  sch\'emas.
\newblock {\em Inst. Hautes \'Etudes Sci. Publ. Math.}, (4):228, 1960.

\bibitem{Har77}
Robin {Hartshorne}.
\newblock {\em {Algebraic geometry.}}
\newblock Springer- Verlag, 1977.

\bibitem{HK07}
I.~Heckenberger and S.~Kolb.
\newblock On the {B}ernstein-{G}elfand-{G}elfand resolution for {K}ac-{M}oody
  algebras and quantized enveloping algebras.
\newblock {\em Transform. Groups}, 12(4):647--655, 2007.

\bibitem{Kac}
Victor~G. Kac.
\newblock {\em Infinite-dimensional {L}ie algebras}.
\newblock Cambridge University Press, Cambridge, third edition, 1990.

\bibitem{KP83}
Victor~G. Kac and Dale~H. Peterson.
\newblock Regular functions on certain infinite-dimensional groups.
\newblock In {\em Arithmetic and geometry, {V}ol. {II}}, volume~36 of {\em
  Progr. Math.}, pages 141--166. Birkh\"auser Boston, Boston, MA, 1983.

\bibitem{Kas89}
Masaki Kashiwara.
\newblock The flag manifold of {K}ac-{M}oody {L}ie algebra.
\newblock In {\em Proceedings of the JAMI Inaugural Conference, supplement to
  Amer. J. Math.} the Johns Hopkins University Press, 1989.

\bibitem{Kas90}
Masaki Kashiwara.
\newblock Kazhdan-{L}usztig conjecture for a symmetrizable {K}ac-{M}oody {L}ie
  algebra.
\newblock In {\em The {G}rothendieck {F}estschrift, {V}ol.\ {II}}, volume~87 of
  {\em Progr. Math.}, pages 407--433. Birkh\"auser Boston, Boston, MA, 1990.

\bibitem{Kas91}
Masaki Kashiwara.
\newblock {On crystal bases of the {\$}q{\$}-analogue of universal enveloping
  algebras}.
\newblock {\em Duke Mathematical Journal}, 63(2):465--516, 1991.

\bibitem{Kas94}
Masaki Kashiwara.
\newblock Crystal bases of modified quantized enveloping algebra.
\newblock {\em Duke Math. J.}, 73(2):383--413, 1994.

\bibitem{KS09}
Masaki Kashiwara and Mark Shimozono.
\newblock Equivariant {$K$}-theory of affine flag manifolds and affine
  {G}rothendieck polynomials.
\newblock {\em Duke Math. J.}, 148(3):501--538, 2009.

\bibitem{KT95}
Masaki Kashiwara and Toshiyuki Tanisaki.
\newblock Kazhdan-{L}usztig conjecture for affine {L}ie algebras with negative
  level.
\newblock {\em Duke Math. J.}, 77(1):21--62, 1995.

\bibitem{Kum02}
Shrawan Kumar.
\newblock {\em Kac-{M}oody groups, their flag varieties and representation
  theory}, volume 204 of {\em Progress in Mathematics}.
\newblock Birkh\"auser Boston, Inc., Boston, MA, 2002.

\bibitem{KS14}
Shrawan Kumar and Karl Schwede.
\newblock Richardson varieties have {K}awamata log terminal singularities.
\newblock {\em Int. Math. Res. Not. IMRN}, (3):842--864, 2014.

\bibitem{Lit95}
Peter Littelmann.
\newblock Paths and root operators in representation theory.
\newblock {\em Ann. of Math. (2)}, 142(3):499--525, 1995.

\bibitem{Lit98}
Peter Littelmann.
\newblock Contracting modules and standard monomial theory for symmetrizable
  {K}ac-{M}oody algebras.
\newblock {\em J. Amer. Math. Soc.}, 11(3):551--567, 1998.

\bibitem{Mat88}
Olivier Mathieu.
\newblock Formules de caract\`eres pour les alg\`ebres de {K}ac-{M}oody
  g\'en\'erales.
\newblock {\em Ast\'erisque}, 159--160:1--267, 1988.

\bibitem{Mat89}
Olivier Mathieu.
\newblock Construction d'un groupe de {K}ac-{M}oody et applications.
\newblock {\em Compositio Math.}, 69(1989):37--60, 1989.

\bibitem{Tit82}
Jacques Tits.
\newblock Algebr\`es de {K}ac-{M}oody et groupes associ\'es.
\newblock Annuaire du College de France (1980--1981) 75--87 et (1981--1982)
  91--106.

\bibitem{Zhu17}
Xinwen Zhu.
\newblock An introduction to affine {G}rassmannians and the geometric {S}atake
  equivalence.
\newblock In {\em Geometry of moduli spaces and representation theory}. Amer.
  Math. Soc., Providence, RI, 2017.

\end{thebibliography}
\bibliographystyle{plain}}
\end{document}